\documentclass[12pt]{article}

\usepackage{amsmath,amsthm,amsfonts,amssymb,amscd}

\usepackage{epsfig}
\usepackage{graphicx}
\usepackage{multirow}
\usepackage{array}
\usepackage[all]{xypic}

\newtheorem{theorem}{Theorem}[section]
\newtheorem{proposition}[theorem]{Proposition}
\newtheorem{corollary}[theorem]{Corollary}
\newtheorem{lemma}[theorem]{Lemma}
\newtheorem{remark}[theorem]{Remark}

\newtheorem{example}[theorem]{Example}

\def\ml{\mathcal{C}}
\def\ml1{\mathcal{C}^1}
\def\mlb1{\mathcal{C}_{b}^{1}}
\def\mr{\mathbb{R}}

\def\mc{\mathbb{C}}
\def\mz{\mathbb{Z}}

\def\frk{\frak}

\def\Phi{{\frk n}}

\usepackage{color}
\usepackage{ifthen}

\newboolean{usecolor}
\setboolean{usecolor}{true}

\ifthenelse{\boolean{usecolor}}
{
  \definecolor{colore}{cmyk}{0,1,0.6,0}
  \definecolor{coloregen}{cmyk}{0.7,0,1,0}
  \definecolor{coloresimo}{cmyk}{1,0.6,0,0}
}
{
  \definecolor{colore}{cmyk}{0,0,0,1}
  \definecolor{coloregen}{cmyk}{0,0,0,1}
  \definecolor{coloresimo}{cmyk}{0,0,0,1}
}

\title{Braid groups in complex spaces}

\author{Sandro {\sc Manfredini}\footnote{Department of Mathematics, University of Pisa, manfredi@dm.unipi.it}  \and Saima {\sc Parveen}\footnote{$^{2}$Abdus Salam School of Mathematical Sciences, GC University, Lahore-Pakistan, saimashaa@gmail.com} \and Simona {\sc Settepanella}\footnote{LEM, Scuola Superiore Sant'Anna, Pisa, s.settepanella@sssup.it}}
\begin{document}

\maketitle

\begin{abstract}
We describe the fundamental groups of ordered and unordered $k-$point sets
in $\mc^n$ generating an affine subspace of fixed dimension. 
\end{abstract}

\begin{center}
{\small\noindent{\bf Keywords}:\\
complex space, configuration spaces, \\braid groups.}
\end{center}

\begin{center}
{\small\noindent{\bf MSC (2010)}:
20F36, 52C35, 57M05, 51A20.}
\end{center}

\section{Introduction}

Let $M$ be a manifold and $\Sigma_k$ be the symmetric group on $k$ elements. The \emph{ordered} and \emph{unordered configuration spaces} of $k$ distinct points in $M$, $\mathcal{F}_k(M)=\{(x_1,\ldots,x_k)\in
M^k|x_i\neq x_j,\,\,i\neq j\}$ and $\mathcal{C}_k(M)=\mathcal{F}_{k}(M)/\Sigma _k$, have been widely studied. It is well known that for a simply connected manifold $M$ of dimension $\geq 3$, the \emph{pure braid group}
$\pi_1(\mathcal{F}_k(M))$ is trivial and the \emph{braid group}
$\pi_1(\mathcal{C}_k(M))$ is isomorphic to $\Sigma_k$, while in low dimensions there
are non trivial pure braids. For example, (see \cite{F}) the pure braid group of the plane
$\mathcal{PB}_n$ has the following presentation 
$$
\mathcal{PB}_n=\pi_1(\mathcal{F}_n(\mathbb{C}))\cong\big<\alpha_{ij},\ 
1\leq i<j\leq n\ \big |\, (YB\,3)_n,(YB\,4)_n\big>,
$$
where $(YB\,3)_n$ and $(YB\,4)_n$ are the Yang-Baxter relations: 
$$
\begin{aligned}
&(YB\,3)_n\!:\,&
\alpha_{ij}\alpha_{ik}\alpha_{jk}=\alpha_{ik}\alpha_{jk}\alpha_{ij}=
\alpha_{jk}\alpha_{ij}\alpha_{ik}, \, 1\leq i <j< k\leq n,\ \ \ \ \ \ \ \ \\
&(YB\,4)_n\!:\,&
[\alpha_{kl},\alpha_{ij}]=[\alpha_{il},\alpha_{jk}]=[\alpha_{jl},
\alpha_{jk}^{-1}\alpha_{ik}\alpha_{jk}]=[\alpha_{jl},\alpha_{kl}\alpha_{ik}
\alpha_{kl}^{-1}]=1,\\
& &  1\leq i<j<k<l\leq n ,
\end{aligned}
$$
while the braid group of the plane $\mathcal{B}_n$ has the well known presentation (see \cite{A})
$$
\mathcal{B}_n=\pi_1(\mathcal{C}_n(\mathbb{C}))\cong\big<\sigma_i,\ 1\leq i\leq n-1\ |\ (A)_n\big>,
$$
where $(A)_n$ are the classical Artin relations:
$$
\begin{aligned}
&(A)_n: \sigma_i\sigma_j=\sigma_j\sigma_i,\, 1\leq i<j\leq n-1,\, j-i\geq 2,\\ &
\ \ \ \ \ \ \ \ \: \sigma_i\sigma_{i+1}\sigma_i=\sigma_{i+1}\sigma_i\sigma_{i+1},\,
1\leq i< n-1.
\end{aligned}
$$ 
Other interesting examples are the pure braid group and the braid group of the sphere $S^2\approx\mc P^1$ with presentations (see \cite{B2} and \cite{F})
$$
\pi_1(\mathcal{F}_{n}(\mc P^1))\cong\big<\alpha_{ij}, 1\leq i<j\leq n-1\,\big|%
(YB\,3)_{n-1},(YB\,4)_{n-1}, D_{n-1}^2=1\big>
$$
$$
\pi_1(\mathcal{C}_{n}(\mc P^1))\cong\big<\sigma_i,1\leq i\leq n-1\,\big|%
(A)_{n},\,\,\sigma_1\sigma_2\ldots\sigma_{n-1}^2\ldots\sigma_2\sigma_1=1\big>,
$$
where $D_k=\alpha_{12}(\alpha_{13}\alpha_{23})(\alpha_{14}\alpha_{24}\alpha_{34}) \cdots(\alpha_{1k}\alpha_{2k}\cdots\alpha_{k-1\ k})$.\\
 The inclusion morphisms $\mathcal{PB}_{n}\rightarrow \mathcal{B}_{n}$ are given by (see \cite{B2})
$$
\alpha _{ij}\mapsto
\sigma _{j-1}\sigma _{j-2}\ldots \sigma _{i+1}\sigma_{i}^{2}\sigma_{i+1}^{-1}\ldots \sigma_{j-1}^{-1}\,\,
$$
and due to these inclusions, we can identify the pure braid $D_n$ with
$\Delta _{n}^{2}$, the square of the fundamental Garside braid (\cite{G}).
In a recent paper (\cite{BS}) Berceanu and the second author introduced new configuration spaces.  They stratify the classical configuration spaces $\mathcal{F}_k(\mathbb{C} P^n)$ (resp. $\mathcal{C}_k(\mathbb{C} P^n)$) with complex submanifolds $\mathcal{F}_{k}^{i}(\mathbb{C}P^n)$ (resp. $ \mathcal{C}_{k}^{i}(\mathbb{C}P^n)$)
defined as the ordered (resp. unordered) configuration spaces of all $k$
points in $\mathbb{C} P^n$ generating a projective subspace of dimension $i$. Then they compute the fundamental groups $\pi_1(\mathcal{F}_{k}^{i}(\mathbb{C}P^n))$ and $\pi_1(\mathcal{C}_{k}^{i}(\mathbb{C}P^n))$, proving that the former are trivial and the latter are isomorphic to $\Sigma_k$ except when $i=1$ providing, in this last case, a presentation for both $\pi_1(\mathcal{F}_{k}^{1}(\mathbb{C}P^n))$ and $\pi_1(\mathcal{C}_{k}^{1}(\mathbb{C}P^n))$ similar to those of the braid groups of the sphere.
In this paper we apply the same technique to the affine case, i.e. to $\mathcal{F}_k(\mathbb{C}^n)$ and 
$\mathcal{C}_k(\mathbb{C}^n)$, showing that the situation is similar except in one case. More precisely we prove that, if $\mathcal{F}_{k}^{i,n}=\mathcal{F}_{k}^{i}(\mathbb{C}^n)$ and $\mathcal{C}_{k}^{i,n}=\mathcal{C}_{k}^{i}(\mathbb{C}^n)$ denote, respectively, the ordered and unordered configuration spaces of all $k$ points in $\mathbb{C}^n$ generating an affine subspace of dimension $i$, then the following theorem holds:
\begin{theorem}  
The spaces $\mathcal{F}_{k}^{i,n}$ are simply connected except for $i=1$ or $i=n=k-1$. In these cases
\begin{enumerate}
\item $\pi_1(\mathcal{F}_{k}^{1,1})=\mathcal{PB}_k$,
\item $\pi_1(\mathcal{F}_{k}^{1,n})=\mathcal{PB}_k/<\,D_{k}\,>$ when $n >1$,
\item $\pi_1(\mathcal{F}_{n+1}^{n,n})=\mathbb Z$ for all $n \geq 1$. 
\end{enumerate}
The fundamental group of $\mathcal{C}_{k}^{i,n}$ is isomorphic to the symmetric group $\Sigma_k$  except for $i=1$ or $i=n=k-1$. In these cases:
\begin{enumerate}
\item $\pi_1(\mathcal{C}_{k}^{1,1})=\mathcal{B}_k$,
\item $\pi_1(\mathcal{C}_k^{1,n})=\mathcal{B}_k/<\, \Delta^2_k\, >$ when $n>1$,
\item $\pi_1(\mathcal{C}_{n+1}^{n,n})=\mathcal{B}_{n+1}/<{\sigma_1}^2={\sigma_2}^2=\cdots={\sigma_n}^2>$ for all $n \geq 1$.
\end{enumerate}
\end{theorem}
\bigskip

\noindent Our paper begins by defining a geometric fibration that connects the spaces $\mathcal{F}_{k}^{i,n}$ to the affine grasmannian manifolds $Graff^{i}(\mc^n)$. In Section \ref{s:due} we compute the fundamental groups for two special cases: points on a line $\mathcal{F}_{k}^{1,n}$ and points in general position $\mathcal{F}_{k}^{k-1,n}$. Then, in Section \ref{s:main}, we describe an open cover of 
$\mathcal{F}_{k}^{n,n}$ and, using a Van-Kampen argument, we prove the main result for the ordered configuration spaces. In Section \ref{five} we prove the main result for the unordered configuration spaces.

\section{Geometric fibrations on the affine \\ grassmannian manifold} 
\label{s:uno}

We consider $\mc^n$ with its affine structure. If $p_1,\ldots,p_k\in\mc^n$ we write \\ $<p_1,\ldots,p_k>$ for the affine subspace generated by $p_1,\ldots,p_k$. We stratify the
configuration spaces $\mathcal{F}_k(\mc^n)$ with complex submanifolds as follows:
$$
\mathcal{F}_k(\mc^n)=\mathop{\coprod}\limits_{i=0}^{n} \mathcal{F}_{k}^{i,n}\,\,,
$$
where $\mathcal{F}_{k}^{i,n}$ is the ordered configuration space of all $k$
distinct points $p_1,\ldots,p_k$ in $\mc^n$ such that the dimension dim$<p_1,\ldots,p_k>=i$.

\begin{remark}\label{rm:1} The following easy facts hold:
\begin{enumerate}
\item ${\mathcal{F}_{k}^{i,n}}\neq \emptyset$ if and only if $i\leq
\min(k+1,n)$; so, in order to get a non empty set, $i=0$ forces $k=1$, and $\mathcal{F}_{1}^{0,n}=\mc^n$.

\item ${\mathcal{F}_{k}^{1,1}}=\mathcal{F}_k(\mc)$,\,\,\,${\mathcal{F}_{2}^{1,n}}=\mathcal{F}_2(\mc^n)$;

\item the adjacency of the strata is given by
$$
\overline{\mathcal{F}_{k}^{i,n}}=\mathcal{F}_{k}^{1,n}\coprod\ldots\coprod\mathcal{F}_{k}^{i,n}.
$$
\end{enumerate}
\end{remark}

By the above remark, it follows that the case $k=1$ is trivial, so from now on we will consider  $k>1$ (and hence $i>0$).

For $i\leq n$, let $Graff^{i}(\mc^n)$ be the affine grassmannian manifold pa\-ra\-me\-tri\-zing $i$-dimensional affine subspaces of $\mc^n$.\\
 We recall that the map $Graff^{i}(\mc^n)\to Gr^i(\mc^n)$ which sends an affine subspace to its direction, exibits $Graff^{i}(\mc^n)$ as a vector bundle over the ordinary grassmannian manifold $Gr^i(\mc^n)$ with fiber of dimension $n-i$. Hence, dim$Graff^i(\mc^n)=(i+1)(n-i)$ and it has the same homotopy groups as $Gr^i(\mc^n)$. In particular, affine grassmannian manifolds are simply connnected and $\pi_2(Graff^i(\mc^n))\cong\mz$ if $i<n$ (and trivial if $i=n$). We can also identify a generator for $\pi_2(Graff^i(\mc^n))$ given by the map
$$g:(D^2,S^1)\rightarrow(Graff^{i}(\mc^n), L_1),\ \ \ g(z)=L_z
$$
where $L_z$ is the linear subspace of $\mc^n$ given by the equations $$(1-|z|)X_1-zX_2 =X_{i+2}=\cdots=X_n=0 \,\,\, .$$  
Affine grasmannian manifolds are related to the spaces $\mathcal{F}_{k}^{i,n}$
through the following fibrations. 

\begin{proposition}
\label{pr:1} The projection
$$
\gamma:\mathcal{F}_{k}^{i,n}\rightarrow Graff^i(\mc^n)
$$
given by
$$
(x_1,\ldots,x_k)\mapsto <x_1,x_2,\ldots,x_k>
$$
is a locally trivial fibration with fiber $\mathcal{F}_{k}^{i,i}$.
\end{proposition}

\noindent\textit{Proof}. Take $V_0\in Graff^i(\mc^n)$ and choose 
$L_{0}\in Gr^{n-i}(\mc^n)$ such that $L_0$ intersects $V_0$ in one point and
 define $\mathcal{U}_{L_0}$, an open neighborhood of $V_0$, by
$$
\mathcal{U}_{L_0}=\{V\in Graff^i(\mc^n)|\ L_0\ \hbox{intersects}\  V\ \hbox{in one point}\}.
$$

For $V\in \mathcal{U}_{L_0}$, define the affine isomorphism
$$
\varphi_V:V\rightarrow V_0,\,\,\varphi_V(x)=(L_0+x)\cap V_0.
$$

The local trivialization is given by the homeomorphism
$$
f:\gamma^{-1}(\mathcal{U}_{L_0})\rightarrow \mathcal{U}_{L_0}\times \mathcal{F}_k^{i,i}(V_0)
$$
$$
y=(y_1,\ldots,y_k)\mapsto\big(\gamma(y),(\varphi_{\gamma(y)}(y_1),\ldots,\varphi_{\gamma(y)}(y_k))\big)
$$
making the following diagram commute (where $\mathcal{F}_k^{i,i}(V_0)=\mathcal{F}_{k}^{i,i}$ upon choosing a coordinate system in $V_0$)

\begin{center}
\begin{picture}(360,120)
\thicklines
\put(77,90){\vector(2,-1){45}}
\put(107,100){\vector(1,0){70}}
\put(197,90){\vector(-2,-1){45}}
\put(55,98){${\gamma}^{-1}(\mathcal{U}_{L_0})$}
\put(184,98){$\mathcal{U}_{L_0}\times \mathcal{F}_{k}^{i,i}$}
\put(130,60){$ \mathcal{U}_{L_0}$}
\put(83,70){$\gamma$}
\put(180,70){$pr_1$}
\put(134,105){$f$}
\put(350,60){$\square$}
\end{picture}
\end{center}

\vspace{-1.3cm}

\begin{corollary}
The complex dimensions of the strata are given by
\end{corollary}

\vspace{-0.7cm}
$$
\hbox{dim}(\mathcal{F}_{k}^{i,n})=\hbox{dim}(\mathcal{F}_{k}^{i,i})+\hbox{dim}( Graff^{i}(\mc^n))=ki+(i+1)(n-i).
$$

\begin{proof}
$\mathcal{F}_{k}^{i,i}$ is a Zariski open subset in $(\mc^i)^k$ for
$k\geq i+1$.
\end{proof}

The canonical embedding
$$
\mathbb{C} ^{m}\longrightarrow\mathbb{C}^n,\,\,\,\,\,\{z_0,\ldots,z_m\}%
\mapsto\{z_0,\ldots,z_m,0,\ldots,0\}
$$
induces, for $i\leq m$, the following commutative diagram of fibrations

\begin{picture}(300,100)
\thicklines
\put(215,75){\vector(0,-1){30}}
\put(147,80){\vector(1,0){40}}
\put(127,75){\vector(0,-1){30}}
\put(56,75){\vector(0,-1){30}}
\put(140,35){\vector(1,0){40}}
\put(75,80){\vector(1,0){40}}
\put(75,35){\vector(1,0){40}}
\put(120,80){$\mathcal{F}_{k}^{i,m}$}
\put(120,30){$\mathcal{F}_{k}^{i,n}$}
\put(190,80){$ Graff^i(\mc ^{m})$}
\put(190,30){$Graff^i(\mc ^n)$}
\put(50,80){$\mathcal{F}_{k}^{i,i}$}
\put(50,30){$\mathcal{F}_{k}^{i,i}$}
\end{picture}

\noindent which gives rise, for $i<m$, to the commutative diagram of
homotopy groups \vspace{-0.7cm}

\begin{picture}(300,120)
\thicklines
\put(230,75){\vector(0,-1){30}}
\put(167,80){\vector(1,0){30}}
\put(137,75){\vector(0,-1){30}}
\put(80,75){\vector(0,-1){30}}
\put(167,35){\vector(1,0){30}}
\put(85,80){\vector(1,0){30}}
\put(85,35){\vector(1,0){30}}
\put(250,80){\vector(1,0){30}}
\put(250,35){\vector(1,0){30}}
\put(120,80){$\pi_1(\mathcal{F}_{k}^{i,i}$)}
\put(120,30){$\pi_1(\mathcal{F}_{k}^{i,i}$)}
\put(200,80){$\pi_1(\mathcal{F}_{k}^{i,m}$)}
\put(200,30){$\pi_1(\mathcal{F}_{k}^{i,n}$)}
\put(285,76){1}
\put(285,32){1}
 \put(35,80){\vector(1,0){30}}
\put(35,35){\vector(1,0){30}} \put(15,78){$\ldots$}
\put(15,33){$\ldots$} \put(75,80){$\mz$}
\put(73,30){$\mz$} \put(85,55){$\cong$} \put(140,55){$\cong$}
\end{picture}

\noindent where the leftmost and central vertical homomorphisms are isomorphisms.\\
Then, also the rightmost vertical homomorphisms are isomorphisms, and we have
\begin{equation}\label{eq:1}
\pi_1(\mathcal{F}_{k}^{i,n})\cong\pi_1(\mathcal{F}_{k}^{i,m})\cong\pi_1(\mathcal{F}_{k}^{i,i+1})\ {\rm for}\ i<m\leq n.
\end{equation}
Thus, in order to compute $\pi_1(\mathcal{F}_{k}^{i,n})$ we can restrict to the case $k \geq n$ (note that $k>i$), computing the fundamental groups $\pi_1(\mathcal{F}_{k}^{i,i+1})$, and for this we can use the homotopy exact sequence of the fibration from Proposition \ref{pr:1}, which leads us to compute the fundamental groups $\pi_1(\mathcal{F}_{k}^{i,i})$.
This is equivalent, simplifying notations, to compute  $\pi_1(\mathcal{F}_{k}^{n,n})$ when $k \geq n+1$.\\
We begin by studying two special cases, points on a line and points in general position.
\section{Special cases}\label{s:due}
\paragraph{The case $i=1$, points on a line.\\[1em]}
 By remark \ref{rm:1} the space ${\mathcal{F}_{k}^{1,1}}=\mathcal{F}_k(\mc)$ for all $k \geq 2$ and the fibration in Proposition \ref{pr:1} gives rise to the exact sequence
\begin{equation}\label{eq:1bis}
\mz=\pi_2(Graff^1(\mc^2))\mathop{\longrightarrow}\limits^{\delta_*}\mathcal{PB}_n=\pi_1(\mathcal{F}_{k}(\mc))\to\pi_1(\mathcal{F}_{k}^{1,2})\to 1\ .
\end{equation}
It follows that $\pi_1(\mathcal{F}_{k}^{1,2})\cong \mathcal{PB}_n/{\rm Im}\delta_*$.
Since $\pi_2(Graff^1(\mc^2))=\mz$, we need to know the image of a generator of this group in $\mathcal{PB}_n$.
Taking as generator the map
$$
g:(D^2,S^1)\to (Graff^1(\mc^2),L_1),\ \ g(z)=L_z,
$$ 
where $L_z$ is the line of equation $(1-|z|)X_1=zX_2$, we chose the lifting
$$
\tilde g:(D^2,S^1)\to (\mathcal{F}_{k}^{1,2},\mathcal{F}_{k}(L_1))
$$ 
$$
\tilde g(z)=\big((z,1-|z|),2(z,1-|z|),\ldots,k(z,1-|z|)\big)
$$
whose restriction to $S^1$ gives the map
$$
\gamma:S^1\longrightarrow \mathcal{F}_{k}(L_1)=\mathcal{F}_{k}(\mathbb{C})
$$
$$
\gamma(z)=((z,0),(2z,0),\ldots,(kz,0))
$$

\begin{lemma}(see \cite{BS})
\label{l.6} The homotopy class of the map $\gamma$
corresponds to the following pure braid in $\pi_1(\mathcal{F}_{k}(\mathbb{C
}))$:
$$
[\gamma]=\alpha_{12}(\alpha_{13}\alpha_{23})\ldots(\alpha_{1k}\alpha_{2k}%
\ldots\alpha_{k-1,k})=D_{k}\,\,.
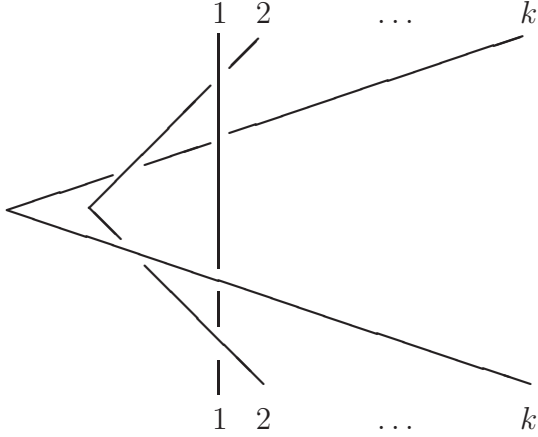
$$
\end{lemma}

\begin{picture}(250,180)

\thicklines

\put(90,58){\line(0,1){89}}
\put(90,36){\line(0,1){13}}
\put(90,10){\line(0,1){13}}
\put(10,80){\line(3,-1){198}}
\put(10,80){\line(3,1){40}}
\put(62,97){\line(3,1){25}}
\put(94,109){\line(3,1){111}}
\put(41,81){\line(1,1){46}}
\put(94,134){\line(1,1){11}}
\put(41,81){\line(1,-1){12}}
\put(62,59){\line(1,-1){45}}
\put(87.5,151){1}\put(104,151){2}\put(204,151){$k$}\put(150,151){$\dots$}

\put(87.5,-3){1}\put(104,-3){2}\put(204,-3){$k$}\put(150,-3){$\dots$}
\end{picture}
\vspace{1cm}

From the above Lemma and the exact sequence in (\ref{eq:1bis}) we get that the image in $\pi_1(\mathcal{F}_{k}(\mc))$ of the generator of $\pi_2(Graff^1(\mc^2))$ is $D_k$ and the following theorem is proved.

\begin{theorem} For $n>1$, the fundamental group of the configuration space of $k$ distinct points in $\mc^{n}$ lying on a line has the following presentation (not depending on $n$)
$$\pi_1(\mathcal{F}_{k}^{1,n})
=\big<\alpha_{ij},\ 1\leq i<j\leq k\ \big|\ (YB3)_k,(YB4)_k,\, D_{k}=1\big>\ .$$
\end{theorem}
\bigskip

\paragraph{The case $k=i+1$, points in general position.}

\begin{lemma}
\label{l.7} For $1 < k\leq n+1$, the projection
$$
p:\mathcal{F}_{k}^{k-1,n}\longrightarrow \mathcal{F}_{k-1}^{k-2,n}
,\,\,\,\,(x_1,\ldots,x_{k})\mapsto(x_1,\ldots,x_{k-1})
$$
is a locally trivial fibration with fiber $\mc^n\setminus \mc^{k-2}$
\end{lemma}

\begin{proof}
Take $(x_{1}^0,\ldots,x_{k-1}^0)\in \mathcal{F}_{k-1}^{k-2,n}$ and fix
$x_{k}^0,\ldots,x_{n+1}^0\in\mc^n$ such that $<x_{1}^0,\ldots,x_{n+1}^0>=\mc^n$ (that is $<x_{k}^0,\ldots,x_{n+1}^0>$ and $<x_{1}^0,\ldots,x_{k-1}^0>$ are skew subspaces). Define the open neighbourhood $\mathcal{U}$ of $(x_{1}^0,\ldots,x_{k-1}^0)$ by
$$
\mathcal{U}=\{(x_1,\ldots,x_{k-1})\in \mathcal{F}_{k-1}^{k-2,n}|\,\,<x_1,\ldots,x_{k-1},x_{k}^0,\ldots,x_{n+1}^0>=\mc^n\}.
$$
For $(x_1,\ldots,x_{k-1})\in\mathcal{U}$, there exists a unique affine isomorphism $T_{(x_1,\ldots,x_{k-1})}:\mc^n\longrightarrow\mc^n,$ which depends continuously on $(x_1,\ldots,x_{k-1})$, such that
$$
T_{(x_1,\ldots,x_{k-1})}(x_{i}^0)=(x_i)\ {\rm for}\ i=1,\ldots,k-1
$$
and
$$
T_{(x_1,\ldots,x_{k-1})}(x_{i}^0)=(x_{i}^0)\ {\rm for}\ i=k,\ldots,n+1 \,\,\, .
$$
We can define the homeomorphisms $\varphi,\psi$ by :
$$
p^{-1}(\mathcal{U})\mathop{\longleftarrow}\limits_{\psi}^{
\mathop{\longrightarrow}\limits^{\varphi}} \mathcal{U}\times \big(\mc^n\setminus <x_1^0,\ldots,x_{k-1}^0>\big)
$$
$$
\varphi(x_1,\ldots,x_{k-1},x)=((x_1,\ldots,x_{k-1}),T_{(x_1,\ldots,x_{k-1})}^{-1}(x))
$$
$$
\psi((x_1,\ldots,x_{k-1}),y)=(x_1,\ldots,x_{k-1},T_{(x_1,\ldots,x_{k-1})}(y))
$$
satisfying $pr_1\circ\varphi=p$.  
\begin{center}
\begin{picture}(360,120)
\thicklines
\put(75,90){\vector(2,-1){50}}
\put(107,100){\vector(1,0){70}}
\put(175,95){\vector(-1,0){70}}
\put(205,90){\vector(-2,-1){50}}
\put(59,98){$p^{-1}(\mathcal{U})$}
\put(184,98){$\mathcal{U}\times \big(\mc^n\setminus <x_1^0,\ldots,x_{k-1}^0>$\big)}
\put(137,56){$ \mathcal{U}$}
\put(85,70){$p$}
\put(190,70){$pr_1$}
\put(134,105){$\varphi$}
\put(134,85){$\psi$}
\end{picture}
\end{center}
\vspace{-2.3cm}
\end{proof}

As $\mc^n\setminus \mc^{k-2}$ is simply connected when $n>k-1$ and $k>1$, we have
$$
\pi_1(\mathcal{F}_{k}^{k-1,n})\cong\pi_1(\mathcal{F}_{k-1}^{k-2,n})\cong\pi_1(\mathcal{F}_{2}^{1,n})=\pi_1(\mathcal{F}_2(\mc^n))\cong\pi_1(\mathcal{F}_{1}^{0,n})=\pi_1(\mc^n)=0,
$$
in particular $\pi_1(\mathcal{F}_{n}^{n-1,n})=0$.
Moreover, since $\mc^n\setminus \mc^{k-2}$ is homotopically equivalent to an odd dimensional (real) sphere $S^{2(n-k)-1}$, its second homotopy group vanish and we have  
$$
\pi_2(\mathcal{F}_{k+1}^{k,n})\cong\pi_2(\mathcal{F}_{k}^{k-1,n})\cong\pi_2(\mathcal{F}_{1}^{0,n})=\pi_2(\mc^n)=0.
$$
in particular $\pi_2(\mathcal{F}_{n}^{n-1,n})=0$.\\
In the case $k=n+1$, $\mc^n\setminus \mc^{n-1}$ is homotopically equivalent to $\mc^*$, and we obtain the exact sequence:
$$
\pi_2(\mathcal{F}_{n}^{n-1,n})\rightarrow\mz\rightarrow\pi_1(\mathcal{F}_{n+1}^{n,n})\rightarrow\pi_1(\mathcal{F}_{n}^{n-1,n})\rightarrow 0.
$$
By the above remarks, the leftmost and rightmost groups are trivial, so we have that
$\pi_1(\mathcal{F}_{n+1}^{n,n})$ is infinite cyclic.\\
We have proven the following
\begin{theorem} For $n\geq 1$, the configuration space of $k$ distinct points in $\mc^{n}$ in general position $\mathcal{F}_{k}^{k-1,n}$ is simply connected except for $k=n+1$ in which case
$\pi_1(\mathcal{F}_{n+1}^{n,n})=\mz$.
\end{theorem}
We can also identify a generator for $\pi_1(\mathcal{F}_{n+1}^{n,n})$ via the map
\begin{equation}\label{eq:4}
h:S^1\rightarrow\mathcal{F}_{n+1}^{n,n}\ \ \ h(z)=(0,e_1,\ldots e_{n-1},ze_n),
\end{equation}
where $e_1,\ldots e_n$ is the canonical basis for $\mc^n$ (i.e. a loop that goes around the hyperplane $<0,e_1,\ldots e_{n-1}>$).


\section{The general case}
\label{s:main}

From now on we will consider $n,i>1$.\\ 
Let us recall that, by Proposition \ref{pr:1} and equation (\ref{eq:1}), in order to compute the fundamental group of the general case $\mathcal{F}_{k}^{i,n}$, we need to compute  $\pi_1(\mathcal{F}_{k}^{n,n})$ when $k \geq n+1$. To do this, we will cover $\mathcal{F}_{k}^{n,n}$ by open sets with an infinite cyclic fundamental group and then we will apply the Van-Kampen theorem to them. 

\subsection{A \textit{good} cover}
Let $\mathcal{A}=(A_1,\ldots,A_p)$ be a sequence of subsets of $\{1,\ldots,k\}$ and the integers $d_1,\ldots,d_p$ given by
$d_j=|A_j|-1,\,\,\,j=1,\ldots,p$. Let us define
$$
\mathcal{F}_{k}^{\mathcal{A},n}=\{(x_1,\ldots,x_k)\in\mathcal{F}_{k}(\mc^n)\big|\dim<x_i>_{i\in A_j}=d_j\ {\rm for\ }j=1,\ldots ,p \}.
$$

\begin{example} The following easy facts hold:
\label{ex:2.6}
\begin{enumerate}
\item If $\mathcal{A}=\{A_1\},\,\,A_1=\{1,\ldots,k\}$, then $\mathcal{F}
_{k}^{\mathcal{A},n}=\mathcal{F}_{k}^{k-1,n};$

\item if all $A_i$ have cardinality $|A_i|\leq 2$, then $\mathcal{F}_{k}^{
\mathcal{A},n}=\mathcal{F}_k(\mathbb{C}^n);$

\item if $p\geq2$ and $|A_p|\leq 2$, then $\mathcal{F}_k^{(A_1,
\ldots,A_p),n}=\mathcal{F}_k^{(A_1,\ldots,A_{p-1}),n}$;

\item if $p\geq2$ and $A_p\subseteq A_1$, then $\mathcal{F}
_k^{(A_1,\ldots,A_p),n}=\mathcal{F}_k^{(A_1,\ldots,A_{p-1}),n}$;

\item $\bigcup_{j\geq i}\mathcal{F}_{k}^{j,n}=\bigcup_{\mathcal{A}=\{A\},A\in {\binom{\{1,\ldots,k\} }{{i+1}}}}\mathcal{F}_{k}^{\mathcal{A},n}.$
\end{enumerate}
\end{example}

\begin{lemma}
\label{l.6.5} For $A=\{1,\ldots,j+1\}$, $j\leq n$, and $k>j$ the map
$$
P_A:\mathcal{F}_{k}^{(A),n}\rightarrow\mathcal{F}_{j+1}^{j,n},\,\,\,(x_1,%
\dots,x_k)\mapsto(x_1,\ldots,x_{j+1})
$$
is a locally trivial fibration with fiber $\mathcal{F}_{k-j-1}(\mathbb{C}
^n\setminus\{0,e_1,\ldots,e_{j}\}).$
\end{lemma}

\begin{proof}
Fix $( x_1,\ldots,x_{j+1})\in \mathcal{F}_{j+1}^{j,n}$ and
choose $z_{j+2},\ldots,z_{n+1}\in \mathbb{C}^n$ such that
$<x_1,\ldots,x_{j+1},z_{j+2},\ldots,z_{n+1}>=\mathbb{C}^n$.\\
Define the neighborhood $\mathcal{U}$ of
$(x_1,\ldots,x_{j+1})$ by $$\mathcal{U}=\{(y_1,\ldots,y_{j+1})\in \mathcal{F}_{j+1}^{j,n}|
\,<y_1,\ldots,y_{j+1},z_{j+2},\ldots,z_{n+1}>=\mathbb{C}^n\} \,\,\, .$$ 
There exists a unique affine isomorphism $F_y:\mathbb{C}^n\rightarrow \mathbb{C}^n$, which depends continuously on $y=(y_{1},\ldots,y_{j+1})$, such that
\begin{equation*}
\begin{array}{ll}
&F_y(x_i)= y_i  \mbox{,} \,\,\,\,i=1,\ldots,j+1 \\
&F_y(z_i)=z_i \hbox{,}\,\,\,\,i=j+2,\ldots,n+1%
\end{array}
\end{equation*}
and this gives a local trivialization
$$
f:P_{A}^{-1}(\mathcal{U})\rightarrow \mathcal{U}\times \mathcal{F}_{k-j-1}(
\mathbb{C}^n\setminus\{x_1,\ldots,x_{j+1}\})
$$
$$
(y_1,\ldots,y_{k})\mapsto\big((y_1,\ldots,y_{j+1}),F_{y}^{-1}(y_{j+2}),
\ldots,F_y^{-1}(y_k)\big)
$$

which satisfies $pr_1\circ f=P_A$ .
\end{proof}

Let us remark that $P_A$ is the identity map if $k=j+1$ and the fibration is (globally) trivial  if $j=n$ since $\mathcal{U}=\mathcal{F}_{n+1}^{n,n}$; in this last case
$\pi_1(\mathcal{F}_{k}^{(A),n})=\mz$ (recall that we are considering $n > 1$).\\
Let  $\mathcal{A}=(A_1,\ldots,A_p)$ be a $p$-uple of subsets of cardinalities $|A_j|=d_j+1$, $j=1,\ldots,p$. For any given integer $h\in\{1,\ldots,k\}$, we define a new $p$-uple $\mathcal{A}^{\prime }=(A^{\prime}_1,\ldots,A^{\prime }_p)$ and integers $d^{\prime }_1,\ldots,d^{\prime }_p$ as follows:
$$
A^{\prime }_j=\left\{
\begin{array}{ll}
A_j,\,\,\hbox{if}\,\,h\notin A_j &  \\
A_j\setminus\{h\},\,\,\hbox{if}\,\,h\in A_j &
\end{array}
\right. ,\,\,\,\,d^{\prime }_j=\left\{
\begin{array}{ll}
d_j,\,\,\hbox{if}\,\,h\notin A_j &  \\
d_j-1,\,\,\hbox{if}\,\,h\in A_j &
\end{array}
\right..
$$

The following Lemma holds.

\begin{lemma}
\label{l.2.9} The map
$$
p_h:\mathcal{F}_{k}^{\mathcal{A},n}\rightarrow\mathcal{F}_{k-1}^{\mathcal{A}%
^{\prime },n},\,\, (x_1,\ldots,x_k)\mapsto(x_1,\ldots,\widehat{x_h}%
,\ldots,x_k)
$$
has local sections with path-connected fibers.
\end{lemma}

\begin{proof}
Let us suppose that $h=k$ and $k\in(A_1\cap\ldots\cap A_l)\setminus(
A_{l+1}\cup\ldots \cup A_p)$. Then the fiber of the map $p_k:\mathcal{F}
_{k}^{\mathcal{A},n}\rightarrow\mathcal{F}_{k-1}^{\mathcal{A}^{\prime },n}$
is
$$
p_{k}^{-1}(x_1,\ldots,x_{k-1})\approx \mathbb{C}^n\setminus \big(%
L^{\prime }_{1}\cup\ldots\cup L^{\prime }_{l}\cup\{x_1,\ldots,x_{k-1}\}\big)
$$
where $L^{\prime }_{j}=<x_i>_{i\in A^{\prime }_j}$. Even in the
case when $\dim L_j=n$, we have $\dim L^{\prime }_j<n$, hence the fiber is
path-connected and nonempty. Fix a base point $x=(x_1,\ldots,x_{k-1})\in
\mathcal{F}_{k-1}^{\mathcal{A}^{\prime },n}$ and choose $x_k\in\mathbb{C}^n\setminus(L^{\prime }_1\cup\ldots \cup L^{\prime
}_l\cup\{x_1,\ldots,x_{k-1}\})$. There are neighborhoods $W_j\subset Graff^{d^{\prime }_j}(\mathbb{C}^n)$ of $L^{\prime
}_j\,(j=1,\ldots,l)$ such that $x_k\notin L^{\prime \prime }_j$ if $%
L^{\prime \prime }_j\in W_j$; we take a constant local section
$$
s:W=g^{-1}\big((\mathbb{C}^n\setminus\{x_k\})^{k-1}\times\mathop{\prod}%
\limits_{i=1}^{l}W_i\big)\rightarrow \mathcal{F}_{k}^{\mathcal{A},n}
$$
$$
(y_1,\ldots,y_{k-1})\mapsto(y_1,\ldots,y_{k-1},x_k),
$$
where the continuous map $g$ is given by:
$$
g:\mathcal{F}_{k-1}^{\mathcal{A}^{\prime },n}\rightarrow(\mathbb{C}^n)^{k-1}\times Graff^{d^{\prime }_1}(\mathbb{C}^n)\times\ldots\times
Graff^{d^{\prime }_l}(\mathbb{C}^n)
$$
$$
(y_1,\ldots,y_{k-1})\mapsto(y_1,\ldots,y_{k-1},L^{\prime \prime
}_{1},\ldots,L^{\prime \prime }_l),
$$
and $L^{\prime \prime }_j=<y_i>_{i\in A^{\prime }_{j}}$ for $%
j=1,\ldots,l$.
\end{proof}




\begin{proposition}
\label{l.5} The space $\mathcal{F}_{k}^{\mathcal{A},n}$ is path-connected.
\end{proposition}

\begin{proof}
Use induction on $p$ and $d_1+d_2+\ldots+d_p$. If $p=1
$, use Lemma \ref{l.6.5} and the space $\mathcal{F}_{j+1}^{j,n}$ which is
path-connected. If $A_p$ is not included in $A_1$ and $d_p\geq3$, delete a
point in $A_p\setminus A_1$ and use Lemma \ref{l.2.9} and the fact that if $C$
is not empty and path-connected and $p:B\rightarrow C$ is a surjective continuous map with local sections such that $p^{-1}(y)$ is path-connected for all $y\in C$, then $B$ is path-connected (see \cite{BS}). If $%
A_p\subset A_1$ or $d_p\leq2$, use Example \ref{ex:2.6}, $(3)$ and $(4)$.
\end{proof}

Let $e_1,\ldots,e_n$ be the canonical basis of $\mc^n$ and 
$$M_h=\{(x_1,\ldots,x_h)\in\mathcal{F}_h(\mc^n\setminus\{0,e_1,\ldots,e_n\})|\ x_1\not\in<e_1,\ldots,e_n>\},$$
the following Lemma holds.
\begin{lemma}
The map $$
p_h:M_h\rightarrow(\mc^n)^*\setminus<e_1,\ldots,e_n>
$$ sending $(x_1,\ldots,x_h)\mapsto x_1$, is a locally trivial fibration with fiber\\
$\mathcal{F}_{h-1}(\mc^n\setminus\{0,e_1,\ldots,e_n,e_1+\cdots+e_n\})$.
\end{lemma}
\begin{proof}
Let $G:B^m\to\mr^m$ be the homeomorphism from the open unit $m$-ball to $\mr^m$ given by
$G(x)=\frac{x}{1-|x|}$, (whose inverse is the map $G^{-1}(y)=\frac{y}{1+|y|}$).
For $x\in B^m$ let $\tilde G_x=G^{-1}\circ\tau_{-G(x)}\circ G$ be an homeomorphism of $B^m$, where $\tau_v:\mr^n\to\mr^n$ is the translation by $v$.
$\tilde G_x$ sends $x$ to 0 and can be extended to a homeomorphism of the closure $\overline{B^m}$, by requiring it to be the identity on the $m-1$-sphere
(the exact formula for $\tilde G_x(y)$ is $\frac{(1-|x|)y-(1-|y|)x}{(1-|x|)(1-|y|)+|(1-|x|)y-(1-|y|)x|}$).\\
 We can further extend it to an homomorphism $G_x$ of $\mr^m$ by setting $G_{x}(y)=y$ if $|y|>1$. Notice that $G_x$ depends continuously on $x$.\\
Let $\bar x\in(\mc^n)^*\setminus<e_1,\ldots,e_n>$, fix an open complex ball $B$ in \\$(\mc^n)^*\setminus<e_1,\ldots,e_n>$ centered at $\bar x$ and an affine isomorphism $H$ of $\mc^n$ sending $B$ to the open real $2n$-ball $B^{2n}$. For $x\in B$, define the homeomorphism $F_x$ of $\mc^n$ $F_x=H^{-1}\circ G_{H(x)}\circ H$ which sends $x$ to $\bar x$, is the identity outside of $B$ and depends continuously on $x$.
The result follows from the continuous map
$$
F:p_h^{-1}(B)\to B\times p_h^{-1}(\bar x)
$$
$$
F(x,x_2,\ldots,x_h)=(x,(\bar x,F_x(x_2),\ldots,F_x(x_h)))
$$
(whose inverse is the map $F^{-1}:B\times p_h^{-1}(\bar x)\to p_h^{-1}(B)$,
$F^{-1}(x,(\bar x,x_2,\ldots,x_h))=(x,F^{-1}_x(x_2),\ldots,F^{-1}_x(x_h))$).\\
The fiber $p_h^{-1}(\bar x)$ is homeomorphic to $\mathcal{F}_{h-1}(\mc^n\setminus\{0,e_1,\ldots,e_n,e_1+\cdots+e_n\})$ via an homeomorphism of $\mc^n$ which fixes $0,e_1,\ldots,e_n$ and sends $\bar x$ to the sum $e_1+\ldots+e_n$.
\end{proof}

Thus we have, since $n \geq 2$, $\pi_1(M_h)=\mz$, and we can choose as ge\-ne\-rator the map
$S^1\rightarrow M_h$ sending $z\mapsto (z(e_1+\cdots+e_n),x_2,\ldots,x_h)$ with $x_2,\ldots,x_h$ of sufficient high norm (i.e. a loop that goes round the hyperplane $<e_1,\ldots,e_n>$).

\begin{lemma}
\label{l.6.5bis} For $A=\{1,\ldots,n+1\}$, $B=\{2,\ldots,n+2\}$, and $k>n+1$ the map
$$
P_{A,B}:\mathcal{F}_{k}^{(A,B),n}\rightarrow\mathcal{F}_{n+1}^{n,n},\,\,\,(x_1,%
\dots,x_k)\mapsto(x_1,\ldots,x_{n+1})
$$
is a trivial fibration with fiber $M_{k-n-1}$
\end{lemma}
\begin{proof}
For $x=(x_1,\ldots,x_{n+1})\in\mathcal{F}_{n+1}^{n,n}$ let $F_x$ be the affine isomorphism of $\mc^n$
such that $F_x(0)=x_1$, $F_x(e_i)=x_{i+1}$, for $i=1,\ldots,n$.  The map
$$
\mathcal{F}_{n+1}^{n,n}\times M_{k-n-1}\rightarrow\mathcal{F}_{k}^{(A,B),n}
$$
sending 
$$
((x_1,\ldots,x_{n+1}),(x_{n+2},\ldots,x_k))\mapsto (x_1,\ldots,x_{n+1},F_x(x_{n+2}),\ldots,F_x(x_k))
$$ gives the result.
\end{proof}

\subsection{Computation of the fundamental group}

From Lemma \ref{l.6.5bis} it follows that $\pi_1(\mathcal{F}_{k}^{(A,B),n})=\mz\times\mz$ and that it has two generators: $((z+1)(e_1+\ldots+e_n),e_1,\ldots,e_n,e_1+\ldots+e_n,x_{n+3},\ldots,x_k)$ and $(0,e_1,\ldots,e_n,z(e_1+\ldots+e_n),x_{n+3},\ldots,x_k)$, where $x_{n+3},\ldots,x_k$ are chosen \textit{far enough} to be different from the first $n+2$ points. 
The first generator is the one coming from the base, the second is the one from the fiber of the fibration $P_{A,B}$.\\
Note that using the map
$$
P'_{A,B}:\mathcal{F}_{k}^{(A,B),n}\rightarrow\mathcal{F}_{n+1}^{n,n},\,\,\,(x_1,%
\dots,x_k)\mapsto(x_2,\ldots,x_{n+2})
$$
we obtain the same result and the generator coming from the base here is the one coming from the fiber above and vice versa.\\
The map $P_{A,B}$ factors through the inclusion $i_A:\mathcal{F}_{k}^{(A,B),n}\hookrightarrow\mathcal{F}_{k}^{(A),n}$ followed by the map
$$
P_A:\mathcal{F}_{k}^{(A),n}\rightarrow\mathcal{F}_{n+1}^{n,n},\,\,\,(x_1,%
\dots,x_k)\mapsto(x_1,\ldots,x_{n+1})
$$ and we get the following commutative diagram of fundamental groups:

\begin{center}
\begin{picture}(360,120)
\thicklines
\put(70,90){\vector(2,-1){50}}
\put(107,100){\vector(1,0){70}}
\put(157,65){\vector(2,1){50}}
\put(45,98){$\pi_1(\mathcal{F}_{k}^{(A,B),n})$}
\put(184,98){$\pi_1(\mathcal{F}_{n+1}^{n,n})$}
\put(112,53){$\pi_1(\mathcal{F}_{k}^{(A),n}) \hspace{5cm} .$}
\put(85,70){${i_A}_*$}
\put(190,70){${P_{A}}_*$}
\put(134,105){${P_{A,B}}_*$}
\end{picture}
\end{center}
\vspace{-1.7cm}

Since $P_A$ induces an isomorphism on the fundamental groups, this means that ${i_A}_*$ sends the generator of $\pi_1(\mathcal{F}_{k}^{(A,B),n})$ coming from the fiber to 0 in $\pi_1(\mathcal{F}_{n+1}^{n,n})$. That is, the generator of $\pi_1(\mathcal{F}_{k}^{(B),n})$ (which is homotopically equi\-va\-lent to the generator of $\pi_1(\mathcal{F}_{k}^{(A,B),n})$ coming from the fiber) is trivial in $\pi_1(\mathcal{F}_{k}^{(A),n})$ and (given the symmetry of the matter) vice versa.\\
Applying Van Kampen theorem, we have that $\mathcal{F}_{k}^{(A),n}\cup\mathcal{F}_{k}^{(B),n}$ is simply connected. Moreover the intersection of any number of $\mathcal{F}_{k}^{(A),n}$'s is path connected and the same is true for the intersection of two unions of $\mathcal{F}_{k}^{(A),n}$'s since the intersection $\bigcap_{A\in {\binom{\{1,\ldots,k\} }{{n+1}}}}\mathcal{F}_{k}^{(A),n}$ is not empty.\\

From the last example in \ref{ex:2.6} with $i=n$ we have $\mathcal{F}_{k}^{n,n}=\bigcup_{A\in {\binom{\{1,\ldots,k\} }{{n+1}}}}\mathcal{F}_{k}^{(A),n}$, and when $k>n+1$, we can cover it with a finite number of simply connected open sets with path connected intersections, so it is simply connected by the following

\begin{lemma} Let $X$ be a topological space which has a finite open cover $U_1,\ldots,U_n$
such that each $U_i$ is simply connected, $U_i\cap U_j$ is connected for all $i,j=1,\ldots,n$ and
$\bigcap_{i=1}^n U_i\neq \emptyset$. Then $X$ is simply connected.
\end{lemma}
\begin{proof}
By induction, let's suppose $\bigcup_{i=1}^{k-1} U_i$ is simply connected. Then, applying Van Kampen theorem to $U_k$ and $\bigcup_{i=1}^{k-1} U_i$, we get that $\bigcup_{i=1}^k U_i$ is simply connected if $U_k\cap(\bigcup_{i=1}^{k-1} U_i)$ is connected. But
$U_k\cap(\bigcup_{i=1}^{k-1} U_i)=\bigcup_{i=1}^{k-1}(U_k\cap U_i)$ is the union of connected sets with non empty intersection, and therefore is connected.
\end{proof}
Now, using the fibration in Proposition \ref{pr:1} with $n=i+1$, we obtain that 
$\mathcal{F}_{k}^{n-1,n}$ is simply connected when $k>n$.\\
Summing up the results for the oredered case, the following main theorem is proved

\begin{theorem} The spaces $\mathcal{F}_{k}^{i,n}$ are simply connected except
\begin{enumerate}
\item $\pi_1(\mathcal{F}_{k}^{1,1})=\mathcal{PB}_k$,
\item $\pi_1(\mathcal{F}_{k}^{1,n})=\big<\alpha_{ij}\,,\,1\leq i<j\leq k\big|(YB3)_k,(YB4)_k,\,D_{k}=1\big>$ when $n >1$,
\item $\pi_1(\mathcal{F}_{n+1}^{n,n})=\mathbb Z$ for all $n \geq 1$, with generator described in (\ref{eq:4}). 
\end{enumerate}
\end{theorem}

\section{The unordered case:  $\mathcal{C}_k^{i,n}$} 
\label{five}

Let $\mathcal{C}_k^{i,n}$ be the unordered configuration space of all $k$ distinct points $p_1,\ldots,p_k$ in $\mathbb C^n$ which generate an $i$-dimensional space.  Then $\mathcal{C}_k^{i,n}$ is obtained quotienting $\mathcal{F}_k^{i,n}$ by the action of the symmetric group $\Sigma_k$. The map $p:\mathcal{F}_k^{i,n}\to\mathcal{C}_k^{i,n}$ is a regular covering with $\Sigma_k$ as deck transformation group, so we have the exact sequence:
$$
1\to\pi_1(\mathcal{F}_k^{i,n})\mathop{\longrightarrow}\limits^{p_*}
\pi_1(\mathcal{C}_k^{i,n})\mathop{\longrightarrow}\limits^{\tau}\Sigma_k\to 1
$$
which gives immediately $\pi_1(\mathcal{C}_k^{i,n})=\Sigma_k$ in case $\mathcal{F}_k^{i,n}$ is simply connected.\\
Observe that the fibration in Proposition \ref{pr:1} may be quotiented obtaining a locally trivial fibration
$\mathcal{C}_k^{i,n}\to Graff^i(\mc^n)$ with fiber $\mathcal{C}_k^{i,i}$.\\
This gives an exact sequence of homotopy groups which, together with the one from Proposition \ref{pr:1} and those coming from regular coverings, gives the following commutative diagram for $i<n$:

\begin{center}
\begin{picture}(330,180)
\thicklines

 \put(215,170){$1$}
\put(124,170){$1$}
\put(127,165){\vector(0,-1){20}}
\put(215,165){\vector(0,-1){20}}
\put(127,75){\vector(0,-1){30}}
\put(120,35){$\Sigma_{k}$}
\put(210,35){$\Sigma_{k}$}
  \put(127,28){\vector(0,-1){20}}
  \put(217,28){\vector(0,-1){20}}
   \put(125,-3){$1$}

   \put(215,-3){$1$}
\put(127,125){\vector(0,-1){30}}
\put(65,125){\vector(0,-1){30}}
\put(219,75){\vector(0,-1){30}}
\put(220,125){\vector(0,-1){30}}
\put(160,130){\vector(1,0){30}}
\put(160,85){\vector(1,0){30}}
\put(250,85){\vector(1,0){30}}
\put(250,130){\vector(1,0){30}}
\put(77,130){\vector(1,0){30}}
\put(85,135){$\delta_*$}
\put(85,90){$\delta_{*}'$}
\put(77,85){\vector(1,0){30}}
 \put(25,130){\vector(1,0){30}}
 \put(25,85){\vector(1,0){30}}
  \put(8,130){$\ldots$}
  \put(8,85){$\ldots$}
\put(200,130){$\pi_1(\mathcal{F}_{k}^{i,n})$}
\put(200,80){$\pi_1(\mathcal{C}_{k}^{i,n}$)}
\put(286,125){$ 1$}
\put(286,81){$1$}
\put(115,130){$\pi_1(\mathcal{F}_{k}^{i,i}$)}
\put(115,80){$\pi_1(\mathcal{C}_{k}^{i,i}$)}
\put(63,127){$\mathbb{Z}$}
\put(63,82){$\mathbb{Z}$}
\put(70,110){$\cong$}
\put(155,37){\vector(1,0){45}}
\end{picture}
\end{center}

In case $i=1$, $\mathcal{F}_{k}^{1,1}=\mathcal{F}_k(\mc)$ and $\mathcal{C}_{k}^{1,1}=\mathcal{C}_k(\mc)$, so $\pi_1(\mathcal{F}_k^{1,1})=\mathcal{PB}_k$ and $\pi_1(\mathcal{C}_k^{1,1})=\mathcal{B}_k$, and since Im$\delta_*=<\, D_k\, >\subset\mathcal{PB}_k$,
the left square gives Im$\delta'_*=<\, \Delta^2_k\, >\subset\mathcal{B}_k$,
therefore $\pi_1(\mathcal{C}_k^{1,n})=\mathcal{B}_k/<\, \Delta^2_k\, >$.

For $i=n=k-1$, we have $\pi_1(\mathcal{F}_{n+1}^{n,n})=\mz$, and we can use the exact sequence of the regular covering $p:\mathcal{F}_{n+1}^{n,n}\to\mathcal{C}_{n+1}^{n,n}$ to get a presentation of $\pi_1(\mathcal{C}_{n+1}^{n,n})$.

Let's fix $Q=(0,e_1,\ldots,e_n)\in\mathcal{F}_{n+1}^{n,n}$ and $p(Q)\in\mathcal{C}_{n+1}^{n,n}$ 
as base points and for $i=1,\ldots n$ define $\gamma_i:[0,\pi]\to\mathcal{F}_{n+1}^{n,n}$ to be the (open) path 
$$
\gamma_i(t)=(\frac{1}{2}(e^{i(t+\pi)}+1)e_i,e_1,\ldots,e_{i-1},\frac{1}{2}(e^{it}+1))e_i,e_{i+1}\ldots,e_n)
$$
(which fixes all entries except the first and the $(i+1)$-th and exchanges 0 and $e_i$ by a half rotation in the line $<0,e_i>$).\\
Then $p\circ\gamma_i$ is a closed path in $\mathcal{C}_{n+1}^{n,n}$ and we denote it's homotopy class in $\pi_1(\mathcal{C}_{n+1}^{n,n})$ by $\sigma_i$. Hence $\tau_i=\tau(\sigma_i)$ is the deck transformation corresponding to the transposition $(0,i)$ (we take $\Sigma_{n+1}$ as acting on $\{0,1,\ldots,n\}$) and we get a set of generators for $\Sigma_{n+1}$ satisfying the following relations
$$
\tau_i^2=\tau_i\tau_j\tau_i\tau_j^{-1}\tau_i^{-1}\tau_j^{-1}=1\ {\rm for}\ i,j=1,\ldots,n \, , 
$$
$$
[\tau_1\tau_2\cdots\tau_{i-1}\tau_i\tau_{i-1}^{-1}\cdots\tau_1^{-1},
\tau_1\tau_2\cdots\tau_{j-1}\tau_j\tau_{j-1}^{-1}\cdots\tau_1^{-1}]=1\ {\rm for}\ |i-j|>2 \,.
$$
If we take $T$, the (closed) path in $\mathcal{F}_{n+1}^{n,n}$ in which all entries are fixed except for one which goes round the hyperplane generated by the others counterclockwise, as generator of $\pi_1(\mathcal{F}_{n+1}^{n,n})$, then $\pi_1(\mathcal{C}_{n+1}^{n,n})$ is generated by $T$ and the $\sigma_1,\ldots,\sigma_n$.\\
In order to get the relations, we must write the words $\sigma_i^2$, $\sigma_i\sigma_j\sigma_i\sigma_j^{-1}\sigma_i^{-1}\sigma_j^{-1}$ and $[\sigma_1\sigma_2\cdots\sigma_{i-1}\sigma_i\sigma_{i-1}^{-1}\cdots\sigma_1^{-1},
\sigma_1\sigma_2\cdots\sigma_{j-1}\sigma_j\sigma_{j-1}^{-1}\cdots\sigma_1^{-1}]$ as well as $\sigma_iT\sigma_i^{-1}$ as elements of Ker $\tau=$ Im $p_*$ for all appropriate $i,j.$\\
Observe that the path $\gamma'_i:[\pi,2\pi]\to\mathcal{F}_{n+1}^{n,n}$, 
defined by the same formula as $\gamma_i$, is a lifting of $\sigma_i$ with starting point
$(e_i,e_1,e_2,\ldots,e_{i-1},0,e_{i-1},\ldots,e_n)$ and that $\gamma_i\gamma'_i$ is a closed path in
$\mathcal{F}_{n+1}^{n,n}$ which is the generator $T$ of $\pi_1(\mathcal{F}_{n+1}^{n,n})$ (as you can see by the homotopy $(\frac{\epsilon}{2}(e^{i(t+\pi)}+1)e_i,e_1,\ldots,e_{i-1},\frac{2-\epsilon}{2}(e^{it}+\frac{\epsilon}{2-\epsilon}))e_i,e_{i+1}\ldots,e_n)$, $\epsilon\in[0,1]$, where for $\epsilon=0$ we have the point $e_i$ going round the hyperplane
$<0,e_1,e_2,\ldots,e_{i-1},e_{i+1},\ldots,e_n>$ counterclockwise).\\
Thus we have $p_*(T)=\sigma_i^2$ for all $i=1,\ldots,n$ (and that Im$p_*$ is the center of $\pi_1(\mathcal{C}_{n+1}^{n,n})$).\\
Moreover, it's easy to see, by lifting to $\mathcal{F}_{n+1}^{n,n}$, that the $\sigma_i$ satisfy the relations
$$
\sigma_i\sigma_j\sigma_i\sigma_j^{-1}\sigma_i^{-1}\sigma_j^{-1}=1\ {\rm for}\ i,j=1,\ldots,n
$$
and
$$
[\sigma_1\sigma_2\cdots\sigma_{i-1}\sigma_i\sigma_{i-1}^{-1}\cdots\sigma_1^{-1},
\sigma_1\sigma_2\cdots\sigma_{j-1}\sigma_j\sigma_{j-1}^{-1}\cdots\sigma_1^{-1}]=1\ {\rm for}\ |i-j|>2 \,\, .
$$
We can represent a lifting of $\sigma'_i=\sigma_1\sigma_2\cdots\sigma_{i-1}\sigma_i\sigma_{i-1}^{-1}\cdots\sigma_1^{-1}$ (which gives the deck transformation corresponding to the transposition $(i,i+1)$) by a path which fixes all entries except the $i$-th and the $(i+1)$-th and exchanges $e_i$ and $e_{i+1}$ by a half rotation in the line $<e_i,e_{i+1}>$.\\
We can now change the set of generators by first deleting $T$ and introducing the relations
$$
\sigma_1^2=\sigma_2^2=\cdots=\sigma_n^2
$$ 
and then by choosing the $\sigma'_i$'s instead of the $\sigma_i$'s. Then we get that the generators  $\sigma'_i$'s satisfy the relations
$$
\sigma'_i\sigma'_{i+1}\sigma'_i=\sigma'_{i+1}\sigma'_i\sigma'_{i+1}\ {\rm for}\ i=1,\ldots,n-1,
$$
$$
[\sigma'_i,\sigma'_j]=1\ {\rm for}\ |i-j|>2
$$
and
\begin{equation}\label{eq:2}
{\sigma'_1}^2={\sigma'_2}^2=\cdots={\sigma'_n}^2.
\end{equation}
Namely, $\pi_1(\mathcal{C}_{n+1}^{n,n})$ is the quotient of the braid group $\mathcal{B}_{n+1}$ on $n+1$ strings by relations (\ref{eq:2}) and the following main theorem is proved.

\begin{theorem}
The fundamental groups $\pi_1(\mathcal{C}_{k}^{i,n})$ are isomorphic to the symmetric group $\Sigma_k$  except
\begin{enumerate}
\item $\pi_1(\mathcal{C}_{k}^{1,1})=\mathcal{B}_k$,
\item $\pi_1(\mathcal{C}_k^{1,n})=\mathcal{B}_k/<\, \Delta^2_k\, >$ when $n>1$,
\item $\pi_1(\mathcal{C}_{n+1}^{n,n})=\mathcal{B}_{n+1}/<{\sigma_1}^2={\sigma_2}^2=\cdots={\sigma_n}^2>$ for all $n \geq 1$.
\end{enumerate}
\end{theorem}

\medskip


\begin{thebibliography}{B1}
\bibitem[A]{A} Artin, E. (1947), \emph{Theory of braids}, Ann. of Math. (2)%
\textbf{48},  pp. 101-126.

\bibitem[BS]{BS} Berceanu, B. and Parveen, S. (2012), \emph{Braid groups in complex projective spaces}, Adv. Geom. \textbf{12}, p.p. 269 - 286. 



\bibitem[B2]{B2} Birman, Joan S. (1974), \emph{Braids, Links, and Mapping
Class Groups}, Annals of Mathematics  vol. \textbf{82}, Princeton University
Press.

\bibitem[F]{F} Fadell, E.R, Husseini, S.Y. (2001), \emph{Geometry and
Topology of Configuration Spaces}, Springer Monographs in Mathematics,
Springer-Verlarg Berlin.

\bibitem[G]{G} Garside, F.A. (1969), \emph{The braid groups and other groups}%
, Quat. J. of Math. Oxford, $2^{e}$ ser. \textbf{20}, 235-254.

\bibitem[H]{H} Hatcher, A. (2002), \emph{Algebraic Topology}, Cambridge
University Press.

\bibitem[M1]{M1} Moran, S. (1983), \emph{The Mathematical Theory of Knots
and Braids}, North Holland Mathematics Studies, Vol 82 (Elsevier, Amsterdam).

\bibitem[M2]{M2} Moulton, V. L. (1998), \emph{Vector Braids}, J. Pure Appl.
Algebra, \textbf{131}, no. 3, 245-296.

\end{thebibliography}
\end{document}